\documentclass[reqno,centertags,12pt]{amsart}
\usepackage{amsmath,amsthm,amscd,amssymb,latexsym,verbatim}
\usepackage{tikz}
\usepackage{float}
\usepackage{enumerate}
\usepackage{tikz}
\usepackage{color}

\usepackage{graphicx,epsf,cite}
\usepackage{epstopdf}
\DeclareGraphicsExtensions{.eps}

-\marginparwidth \oddsidemargin -4mm \evensidemargin \oddsidemargin

\newtheorem{theorem}{Theorem}[section]

\newtheorem{proposition}[theorem]{Proposition}
\newtheorem{lemma}[theorem]{Lemma}
\newtheorem{corollary}[theorem]{Corollary}
\theoremstyle{definition}
\newtheorem{definition}[theorem]{Definition}

\theoremstyle{remark}

\newcounter{smalllist}

\DeclareMathOperator{\diam}{diam}

\allowdisplaybreaks
\numberwithin{equation}{section}

\newcommand{\abs}[1]{\left\lvert#1\right\rvert}

\newcommand{\lb}{\label}

\newcommand{\supp}{\text{\rm{supp}}}

\newcommand{\beq}{\begin{equation}}
\newcommand{\eeq}{\end{equation}}

\newcommand{\bal}{\begin{align}}
\newcommand{\eal}{\end{align}}
\newcommand{\bals}{\begin{align*}}
\newcommand{\eals}{\end{align*}}

\newcommand{\bbR}{{\mathbb{R}}}

\newcommand{\calL}{{\mathcal L}}

\newcommand{\eps}{\varepsilon}
\newcommand{\del}{\delta}
\newcommand{\tht}{\theta}
\newcommand{\ka}{\kappa}
\newcommand{\al}{\alpha}
\newcommand{\be}{\beta}
\newcommand{\ga}{\gamma}
\newcommand{\la}{\lambda}

\newcommand{\til}{\tilde}

\newcommand{\mbr}{\mathbb R}

\newcommand{\de}{\delta}

\newcommand{\ta}{\theta}
\newcommand{\si}{\sigma}

\newcommand{\cb}[1]{\left\{#1\right\}}

\newcommand{\rb}[1]{\left(#1\right)}

\newcommand{\enn}[2]{\begin{enumerate}\renewcommand{\theenumi}{#1\arabic{enumi}}
	\vspace{2mm}
	\itemsep0.2em
	\setlength\itemindent{13pt}
	#2
	\vspace{2mm}
\end{enumerate}}

\begin{document}
\title[Non-local KPP Transition Fronts]
{Transition Fronts for Inhomogeneous Fisher-KPP Reactions and Non-local Diffusion}

\author{ Tau Shean Lim and Andrej Zlato\v s}

\address{\noindent Department of Mathematics \\ University of
Wisconsin \\ Madison, WI 53706, USA 
%\newline Email: zlatos@math.wisc.edu
}

\begin{abstract}
We prove existence of and construct transition fronts for a class of reaction-diffusion equations with spatially inhomogeneous Fisher-KPP type reactions and non-local diffusion. Our approach is based on finding these solutions as perturbations of appropriate solutions to the linearization of the PDE at zero.  Our work extends a method introduced by one of us to study such questions in the case of classical diffusion.
%the existence of transition front solutions of inhomogeneous local diffusion reaction equations by constructing super- and sub-solutions. We derive a similar conclusion for nonlocal model with further assumption on reaction function. 
\end{abstract}

\maketitle

%%%%%%%%%%%%%%%%%%%%%%%%%%%%%%%%
\section{Introduction and Main Results}
%%%%%%%%%%%%%%%%%%%%%%%%%%%%%%%%
In this paper we study the existence of transition fronts for a class of reaction-diffusion equations with inhomogeneous  Kolmogorov-Petrovskii-Piskunov (KPP) type nonlinearities (also called Fisher-KPP \cite{Fisher, KPP}) and non-local diffusion. 
%We study the question of existence of transition fronts for these PDE via a method introduced in \cite{ZlaInhomog} in the case of classical diffusion, finding these solutions as perturbations of analogous solutions of the linearization of the PDE at $u=0$.         
%In the present paper we consider non-local reaction-diffusion equations with Fisher-KPP (Kolmogorov-Petrovskii-Piskunov) type nonlinearities. We study the question of existence of transition fronts for these PDE via a method introduced in \cite{ZlaInhomog} in the case of classical diffusion, finding these solutions as perturbations of analogous solutions of the linearization of the PDE at $u=0$.   
We consider the PDE
\begin{equation}
\lb{1.1} u_t=Hu+f(x,u),
\end{equation}
with the {\it non-local diffusion operator}
% $H$ is the nonlocal diffusion operator, a convolution operator, given by
\[
(Hu)(x,t):= (J*u)(x,t)-u(x,t)=\int_{\mbr} J(y)[u(x-y,t)-u(x,t)]dy.
\]
The  kernel $J\in C^1(\mbr)$ 
%is a probability density function on $\mbr$ which 
satisfies on $\bbR$
\enn{J}{
\item \lb{J1} $J\ge 0$ is even and non-increasing on $\mbr^+$; %(?) \comment{Used in \eqref{6.3}};
%, $J(x)=J(-x)\ge 0$, $J$ is decreasing on $\mbr^+$ (and increasing on $\mbr^-$ by symmetry);
\item \lb{J2} $\supp\, J=[-\delta,\delta]$ and $\int_{-\de}^{\de} J(y)dy=1$ (here $\de >0$ need not be small).
}
The  {\it inhomogeneous  KPP reaction function} $f\in C^2(\mbr\times[0,1])$ satisfies on $\mbr\times[0,1]$
% , with $a(x):= f_u(x,0)$, satisfies
\enn{F}{
\item \lb{F1} $f\ge 0$ and $f(x,0)=f(x,1)=0$;
\item \lb{F3} there is $\theta_1\in(0,1)$ such that $f_u(x,u)\le 0$ when $u\in[\theta_1,1]$;
\item \lb{F2} $a(x)g(u)\le f(x,u)\le a(x)u$, with $a(x):= f_u(x,0)$ and some $g$ as below.
}
The function $g\in C^1([0,1])$ in \eqref{F2} satisfies on $[0,1]$
\enn{G}{
\item \lb{G1} $g\ge 0$ and $g(0)=g(1)=0$;
\item \lb{G2}$g'(0)=1$, $g'$ is decreasing, $g'(1)\ge -1$, and $\int_0^1 u^{-2}(u-g(u))du<\infty$. 	
}
We denote $a_-:=\inf_{x\in\mbr} a(x)$ and $a_+:=\sup_{x\in\mbr} a(x)$, and also require 
\begin{equation} \lb{1.2z}
a_->0.
\end{equation}
We have $a_+<\infty$ by $f\in C^1$, and $f$ is of KPP type because $f(x,u)\le f_u(x,0)u$.
%Let us also assume that $a_-, a_+$ are such that
%\begin{equation} \lb{1.2z}
%0< a_-\le \inf_{x\in\mbr} a(x)\le  \sup_{x\in\mbr} a(x) \le a_+< \infty
%\end{equation}
%(the existence of such $a_+$ is obvious from $f\in C^1$).  Finally, note that $f$ is of KPP type because $f(x,u)\le f_u(x,0)u$.

A {\it (right moving) transition front} for \eqref{1.1} is any solution $0\le u\le 1$ on $\mbr\times\mbr$ such that
\begin{equation}
\lb{1.3}\lim_{x\to-\infty} u(x,t)=1 \qquad\text{and}\qquad \lim_{x\to\infty} u(x,t)=0 %\qquad \text{for each $t\in\mbr$,}
\end{equation}
for each $t\in\mbr$, and $u$ has a {\it bounded width}. The latter means that for each $\eps>0$,
% there is $L_\ep<\infty$, such that
\begin{equation}
\lb{1.4} \sup_{t\in\bbR} L_{u,\eps}(t):=\sup_{t\in\mbr} \diam\{x\in \mbr:\eps\le u(x,t)\le 1-\eps\} <\infty. % \le L_\ep.
\end{equation}
This notion of transition fronts is the 1-dimensional case of the definition by Berestycki-Hamel, which was stated for equations with classical diffusion (i.e., $\partial_{xx}$ in place of $H$) in \cite{BH3}.  It is a generalization of the notion of {\it traveling fronts} for homogeneous media and {\it pulsating fronts} for periodic media.  The former are solutions of \eqref{1.1} (or its classical diffusion counterpart) with $f(x,u)=f(u)$, which are of the form $u(x,t)=U(x-ct)$ for some speed $c\in\bbR$ and profile $U:\bbR\to(0,1)$ such that $\lim_{s\to-\infty} U(s)=1$ and $\lim_{s\to\infty} U(s)=0$.  The latter are solutions of \eqref{1.1} with $x$-periodic $f$, which are of the form $u(x,t)=U(x-ct,x)$, with $U$ periodic in and the above limits uniform in the second argument.

Traveling and pulsating fronts in the presence of classical diffusion have been extensively studied, starting with the works of Fisher \cite{Fisher} and Kolmogorov-Petrovskii-Piskunov \cite{KPP}.  Instead of surveying the vast literature, let us refer to the review articles by Berestycki \cite{Berrev} and Xin \cite{Xin2}, and mention specifically  that in the homogeneous/periodic KPP case, there exists a traveling/pulsating front precisely when the speed $c\ge c_f$, where the number $c_f>0$ is the {\it minimal front speed} for $f$ (in the homogeneous case $c_f=2\sqrt{f'(0)}$).  

The corresponding results for the non-local diffusion equation \eqref{1.1} are considerably more recent.  For instance, in \cite{BatesTravelBistable,CarrChmaj,CovDupaNonlocal,CovDupaSpeed,CovUniq}, existence, uniqueness, and other properties of traveling fronts are proved for various kernels $J$ and  various types of homogeneous reactions $f$ (KPP, monostable, ignition, and bistable). 
%In \cite{LiSunWang}, Li, Sun and Wang construct entire solution of \eqref{1.1} by combining two fronts with different speed, echo to Hamel and Nadirashvili's work in local diffusion settings \cite{HamelNadiEntire}. Last but not least, 
The case of periodic KPP reactions was also addressed by Coville, D\'avila, and Mart\'inez in \cite{CovPulsate}, where pulsating fronts were proved to exist precisely when the speed $c\ge c_{J,f}$ (for homogeneous reactions this was proved in \cite{CovDupaNonlocal}). In fact, \cite{CovPulsate} applies in several spatial dimensions, where it proves that for each unit vector $e$ there again exists a pulsating front in direction $e$ with speed $c$ precisely when $c\ge c_{J,f,e}$.  We mention that traveling fronts for equations with non-local diffusion represented by the fractional Laplacian and homogeneous ignition reactions \cite{MRS2}, as well as with classical diffusion and non-local homogeneous KPP reactions \cite{BNPR} were also studied recently.

In these studies, both for classical and non-local diffusion, it has been of crucial help that the traveling front ansatz $u(t,x)=U(x-ct)$ turns the PDE \eqref{1.1} into an ODE.  The pulsating front ansatz $u(t,x)=U(x-ct,x)$ ($U$ periodic in the second argument) similarly yields a degenerate elliptic PDE. 
 For general (non-periodic) inhomogeneous reactions, on the other hand, no such simplification is available.  Because of this, the question of existence and properties of transition fronts for \eqref{1.1} with  classical diffusion and general  inhomogeneous reactions 
%of various types
%(KPP as well as monostable, ignition, and bistable types) 
 has  been addressed only recently in, among other works, \cite{MRS, MNRR, NRRZ, NolRyz, TZZ, VakVol, ZlaInhomog, ZlaGenfronts}.  
 The present paper is, to the best of our knowledge, the first study of the analogous non-local diffusion problem.

Our main result is existence of transition fronts for \eqref{1.1} with KPP reactions whose $a(x)=f_u(x,0)$ is sufficiently close to a constant (while $f$ itself need not be close to a homogeneous reaction).  We prove this by extending to this model a method introduced by one of us in \cite{ZlaInhomog}
%, where such questions were studied in the presence of 
for the classical diffusion case.  
The idea here is to exploit the close relationship between \eqref{1.1} and its linearization at $u=0$, 
\begin{equation}  \label{2.1}
v_t = Hv +a(x)v.
\end{equation}
%Such close relationship exists because all KPP fronts are {\it pulled}, with the front speeds determined by the reaction at $u=0$, which is due to the {\it reaction strength} $\tfrac{f(x,u)}u$ being largest at $u=0$ for any fixed $x\in\bbR$.  This is in stark contrast with ignition fronts, which are always {\it pushed} because they are ``driven'' by the reaction at intermediate values of $u$.  
We will therefore first study the simpler case of {\it front-like solutions} of \eqref{2.1}, of the form
\beq\lb{1.5}
v_{\lambda}(x,t) = e^{\lambda t}\phi_{\lambda}(x).
\eeq
Here $\phi_{\lambda}>0$ is a generalized eigenfunction of the operator $H+a(x)$, satisfying 
\begin{equation} \label{2.2}
H \phi_\lambda + a(x) \phi_{\lambda} = \lambda \phi_{\lambda}
\end{equation}
on $\bbR$, which  grows exponentially  to $\infty$ as $x\to-\infty$ and decays exponentially  to 0 as $x\to\infty$.  

In the case of classical diffusion, Sturm-Liouville theory assures existence of (a unique up to a multiple) such $\phi_\lambda$ if and only if $\lambda>\sup\sigma(\partial_{xx}+a(x))$ (with $\sigma(\calL)$ the spectrum of  $\calL$).   We will prove that for \eqref{2.2}, such $\phi_\lambda$ exists for each $\lambda>a_+$.   Note that $H$ is a  negative operator on $L^2(\bbR)$, so $a(x)\le a_+$ shows $a_+\ge \sup\sigma(H+a(x))$. In fact, $-2I\le H\le 0$, with $I$ the identity operator,  since  $\|J*\phi\|_2\le \|J\|_1\|\phi\|_2=\|\phi\|_2$ by Young's inequality.  

Also note that if $a$ is constant, then $\sup\sigma(H+a)=a$ and  for each $\lambda>a$ there is $p_\la>0$ such that $\phi_\la(x)=e^{-p_{\la}x}$ solves \eqref{2.2}.  This $p_\lambda$ is unique and given by $\int_\bbR J(y)e^{p_\la y}dy=1+\la-a$.  In this case the solution \eqref{1.5} can also be written as $v_\la(x,t)=e^{-p_\la(x-c t)}$, with speed $c=\la p_\la^{-1}$.  In the general inhomogeneous case, however, fronts for \eqref{1.1} and \eqref{2.1} typically do not have specific speeds,  so one cannot anymore ``parametrize'' fronts via their speeds $c$.  Instead, one can  use the ``energies'' $\la$ for this purpose.

Next we note that by  \eqref{F2}, solutions of \eqref{2.1} are super-solutions of \eqref{1.1}.  The main result of \cite{ZlaInhomog} is showing that in the case of classical diffusion, for each $\lambda\in(\sup\sigma(\partial_{xx}+a(x)), 2a_-)$ there is a function $h_\lambda:[0,\infty)\to[0,1)$ such that  $w_\lambda:=h_\lambda(v_\lambda)\le v_\lambda$ is a sub-solution of \eqref{1.1}, and then finding a transition front  $u_\lambda$ for \eqref{1.1} between $w_\lambda$ and $\min\{v_\lambda,1\}$.  This $h_\lambda$ satisfies
\beq\lb{1.32}
h_\lambda(0)=0, \qquad h_\lambda'(0)=1,   \qquad \lim_{v\to\infty} h_\lambda(v)=1, \qquad  \text{and} \qquad h_\lambda''<0 \quad \text{on $(0,\infty)$},
\eeq
which also means that $h_\lambda$ is increasing and $h_\lambda(v)\le v$ on $[0,\infty)$.  From $\lim_{v\to 0} v^{-1}h_\lambda(v)=1$, $\lim_{x\to\infty}v_\la(x,t)=0$ for each $t\in\bbR$, and $w_\lambda\le u_\lambda\le v_\lambda$ it follows that
\beq\lb{1.41}
\lim_{x\to\infty} \frac{u_\lambda(x,t)}{v_\lambda(x,t)}=1
\eeq
for each $t\in\bbR$.  We note that the bound $\lambda<2a_-$ is not just a technical limitation; it is sharp for constant $a$, and there are also examples of KPP $f$ with $\sup\sigma(\partial_{xx}+a(x))> 2a_-$ for which no transition fronts exist at all \cite{NRRZ}.

In the present paper  we show that this approach can be  extended to the non-local diffusion equation \eqref{1.1}.  
%The extension is non-trivial for two reasons.  
To do so, we need to overcome three new difficulties.  First, we are not aware of a version of the Sturm-Liouville theory for operators $H+a(x)$, and have to prove the necessary result below (Lemma \ref{L.2.1}).  Second, due to the non-locality of $H$, we need to obtain very good estimates on the oscillation of the generalized eigenfunctions $\phi_\la$ (Lemma \ref{L.3.2}) in order to apply the (local in nature) method of finding sub-solutions from \cite{ZlaInhomog}.
And third, \eqref{1.1} lacks the regularizing effects of its classical diffusion counterpart.  In fact, the fundamental solution of $u_t=Hu$ is  
\[
\Gamma(x-x_0,t):=e^{-t} \delta_{0}(x-x_0)+e^{-t}(e^{t\hat J}-1)\check{}\, (x-x_0),
\]
 where $\delta_{0}$ is the delta function at $0$ (see  \cite[Lemma 1.6]{VailloNDP}). We overcome this lack of parabolic regularity theory for \eqref{1.1} by showing that while the regularity of solutions of the PDE does not improve with time, for at least some solutions it does not worsen arbitrarily either (Lemma \ref{L.4.1}). 
%develop a different approach. 
% to make up the ``missing piece" of parabolic regularity.
Our main result is as follows.

%We say that two transition fronts are distinct if they are not time-translations of each other.

% % % % % %Main Theorem % % % % % %
\begin{theorem} \lb{T.1.1}
Assume that $J,f,g$ satisfy the hypotheses (J), (F), (G) above and \eqref{1.2z}. 
% There exists a constant $\la_0=\la_0(J,a_+',a_-)$ (what properties?? $>a_-$??), such that if $a_+<\la_0$, then \eqref{1.1} has infinitely many distinct transition fronts (??).

(i)
If $\la>a_+$, then \eqref{2.2} has a continuous solution $\phi_\lambda>0$ with 
\begin{equation} \lb{1.6} 
\lim_{x\to-\infty} \phi_\la(x)=\infty \qquad\text{and}\qquad \lim_{x\to\infty} \phi_\la(x)=0. 
\end{equation}
(In fact, by Lemma \ref{L.2.1}, $\phi_\la$ grows and decays at least exponentially as $x\to -\infty$ and $x\to \infty$.)
Thus  $v_\lambda$ from \eqref{1.5} is a super-solution of \eqref{1.1}.

(ii)
 There are $\la_0=\la_0(J,a_-)>0$ (which is non-decreasing in $a_-$) and $h_g:[0,\infty)\mapsto[0,1)$ satisfying \eqref{1.32} such that if $a_+<a_-+\la_0$, then for each $\la\in (a_+,a_-+\la_0)$ the function $w_\la:= h_g(v_\la)$ is a sub-solution of \eqref{1.1}. 
 %$h$ can be chosen according to \eqref{3.1} with any $\al\in(\al_0,1)$ for some $\al_0$. 

(iii) If $\la\in (a_+,a_-+\la_0)$, then there exists a transition front $u_\la$ for \eqref{1.1} satisfying %\eqref{1.41} 
\begin{equation} \lb{1.7}
w_\la\le u_\la\le \min\cb{v_\la,1}.
\end{equation}
\end{theorem}
% % % % % End Theorem % % % % % % %

\iffalse 
% % % %figure % % % % %
\begin{figure}[H]
\centering
\includegraphics[scale=0.28]{nonlocal_fig1.pdf}
%\begin{tikzpicture}[domain=-4:5,xscale=2, yscale=3, scale=0.55,
%every node/.style={scale=.8},
%declare function={phi(\t)=1+0.08*cos(200*\t))*exp(-0.3*\t);},
%declare function={h(\t)=4/3*(1-1/((1+\t)^(0.75)));}]
%\draw[->] (-4.2,0.14) -- (5.2,0.14) node[right] {$x$};
%\draw[->] (0,-0.2) -- (0,0.14) node[below left]{0} -- (0,3.5);
%\node[below left] (mark) at (0,1) {1};
%\draw[smooth,thick] plot (\x,{phi(\x)});
%\node (super) at (-1,1.6) {$v_\la $};
%\draw[smooth,thick] plot (\x,{h(phi(\x))});
%\node (sup) at (-1,0.4) {$w_\la =h_\la (v_\la )$};
%\draw[dashed] plot (\x,1);
%\end{tikzpicture}
\caption{The super- and sub-solutions at some fixed $t\in \mbr $.}
\end{figure}
% % % % end figure % % % % % % %
\fi

{\it Remarks.} 
1. Obviously \eqref{1.41} holds again.
\smallskip

2.  Here $h_g$ only depends on $g$ only, so not on $\la\in (a_+,a_-+\la_0)$.
\smallskip

3. In the case of classical diffusion \cite{ZlaInhomog} obtains $\la_0=a_-$, which is sharp.  An explicit expression for our $\lambda_0$ can be found from the formulas in Sections \ref{S3} and \ref{S6}, but we do not know what the sharp value is in this case.
%it is unlikely to be sharp. 
%\comment{The expression of bound is complicated, not informative}
%In this paper, we only show the existence of such $\la_0$, but it is not necessarily sharp. 
\smallskip

4. 
At the end of Section \ref{S4} we obtain an explicit upper bound on the $\eps$-width of $u_\la$ (defined in \eqref{1.4}) for $\lambda\in(a_+,a_-+\lambda_0)$. This bound depends only on $\eps,J,g$ and on an upper bound for $a_-$ and $(\la-a_+)^{-1}$.
% (to see this we also note that $\lambda-a_-\le\lambda_0(J,a_-)$). 
\smallskip

5. As can be easily seen from the proof, the theorem extends to time-dependent $f$ such that $f_u(t,x,0)$ is time independent and (F) holds for each $t\in\bbR$. 
\smallskip

6. If  $a_+<\inf_n \lambda_n\le \sup_n \lambda_n < a_-+\la_0$ and $b_n>0$ are such that $\sum_{n} b_n < \infty$, then  
as in \cite{ZlaInhomog}, the result holds with $\phi_\la$ and $v_\lambda$ replaced by $\sum_{n} b_n \phi_{\la_n}$ and $\sum_{n} b_n v_{\la_n}$. The corresponding fronts are a combination of a countable number of the ``pure" fronts from (iii). Their existence is new even in the cases of homogeneous and periodic reactions (and non-local diffusion). 
\smallskip

%4. If $\mu$ is a Borel measure supported on a compact subset of $(a_+,a_-+\la_0)$, then  
%as in \cite{ZlaInhomog}, the result holds with $\phi_\la$ and $v_\lambda$ replaced by $\phi_\mu(x):=\int \phi_\la(x)d\mu(\la)$ and $v_\mu(x,t):=\int v_\la(x,t)d\mu(\la)$.\smallskip

%Theorem \ref{T.1.1} extends the result for local diffusion equation (Laplacian) obtained in \cite{ZlaInhomog} to the case of nonlocal operator, with further assumption on $f$ and $g$. The main difficulty of this generalization comes from the difference between local and nonlocal diffusion. In the case of local diffusion, many analytic results (such as behavior of eigenfunctions and parabolic regularity) have been well-established. In this paper, we have to develop the similar results from scratch. For example, it is well-known that \eqref{1.1} lacks of regularizing effect. In fact, the fundamental solution of $u_t=Hu$ is given by $u(x,t)=e^{-t} \delta_{x_0}(x)+K(x,t)$, where $\delta_{x_0}$ is delta measure concentrates on $x_0$. In order to obtain transition solutions from uniform limit, we develop a new way to obtain global Lipschitz condition to make up the ``missing piece" of parabolic regularity. 

We prove the three parts of Theorem \ref{T.1.1} in the next three sections, postponing the proofs of two crucial estimates needed for the construction of sub-solutions until Sections 5 and 6. 
%, we give the proof of the crucial estimates we need in the proof of theorem \ref{T.1.1}. 
\smallskip

{\bf Acknowledgements.}  We thank the anonymous referee for questions that pushed us to improve our main result.  TSL was partially supported by NSF grant DMS-1056327, and AZ was partially supported by NSF grants DMS-1056327, DMS-1113017,  and DMS-1159133.

%% % % %figure % % % % %
%\begin{figure}[h]
%\begin{tikzpicture}[domain=-4:5,xscale=2, yscale=3, scale=0.5,
%every node/.style={scale=.8},
%declare function={phi(\t)=1+0.1*cos(200*\t))*exp(-0.3*\t);},
%declare function={h(\t)=4/3*(1-1/((1+\t)^(0.75)));}]
%\draw[->] (-4.2,0) -- (5.2,0) node[right] {$x$};
%\draw[->] (0,-0.2) -- (0,0) node[below left]{0} -- (0,3.5);
%\node[below left] (mark) at (0,1) {1};
%\draw[smooth,thick] plot (\x,{phi(\x)});
%\node (super) at (-1,1.6) {$v$};
%\draw[smooth,thick] plot (\x,{h(phi(\x))});
%\node (sup) at (-1,0.4) {$w=h(v)$};
%\draw[dashed] plot (\x,1);
%\end{tikzpicture}
%\centering
%\caption{Graph of super- and sub-solution at a given time}
%\end{figure}
%%$$$$$$$$$$$$$$$$$$$$$$$$$

%%%%%%%%%%%%%%%%%%%%%%%%%
\section{Proof of Theorem \ref{T.1.1}(i) (Construction of a Super-solution)}
% % % % % % % % % % % % %
Recall that $v_\lambda$ from \eqref{1.5} is a super-solution of \eqref{1.1} when $\phi_\lambda$ solves \eqref{2.2}.  We thus only need to prove the following result.

%The linearization of \eqref{1.1} at $u=0$ is the PDE 
%\begin{equation}
%\lb{2.1}v_t=Hv+a(x)v.
%\end{equation}
%By  \eqref{F2}, its solutions are super-solutions of \eqref{1.1}. We will construct solutions of \eqref{2.1} of the form \eqref{1.5}, with $\phi_\lambda$  a generalized eigenfunction of the operator $H+a(x)$, satisfying 
%\begin{equation}
%\lb{2.2}H\phi_\lambda+a(x)\phi_\lambda=\lambda \phi_\lambda
%\end{equation}
%for some $\lambda\in\bbR$, as well as \eqref{1.6}.
%%$\lim_{x\to\infty} \phi_\lambda(x)=0$ and $\lim_{x\to-\infty} \phi_\lambda(x)=\infty$.  
%
%In the case of classical diffusion (i.e., $\partial_{xx}$ in place of $H$), the Sturm-Liouville theory assures that \eqref{2.2} has a unique such solution if and only if $\lambda>\sup\sigma(\partial_{xx}+a(x))$ (with $\sigma(L)$ the spectrum of the operator $L$). In our nonlocal diffusion problem, the same result holds for $\lambda>a_+$ by the following proposition, possibly with the exception of uniqueness. 
%
%We note that $H$ is a bounded negative operator on $L^2(\bbR)$. (In fact, $-2I\le H\le 0$, with $I$ the identity operator,  since  $\|J*\phi\|_2\le \|J\|_1\|\phi\|_2=\|\phi\|_2$ by Young's inequality.)  Hence, $a_+\ge \sup\sigma(H+a(x))$. Also note that if $a$ is constant, then $\sup\sigma(H+a)=a$ and  for each $\lambda>a$ there is a unique $c_\la>0$ such that the function $\phi_\la(x)=e^{-c_{\la}x}$ solves \eqref{2.2}.  This $c_\lambda$ is in fact given by $\int_\bbR J(y)e^{-c_\la y}dy=1+\la-a$. 

% % % % % % % % % %
\begin{lemma}
\lb{L.2.1} If $\lambda>a_+$, then there is a continuous solution $\phi_\lambda>0$ of \eqref{2.2} and $L=L(J,\la-a_-,\la-a_+)>0$ such that $\phi_\la(x)\ge 2\phi_\la(y)$ whenever $y\ge x+L$.
%  and  $\phi_\la(x)\ge A^{-1} e^{-mx}$ for $x\le 0$.
%\[\phi_\la(x)\le Ae^{-mx} \qquad \text{for $x\ge 0$} \]
%and 
%\[ \phi_\la(x)\ge A^{-1} e^{-mx} \qquad \text{for $x\le 0$.} \]
\end{lemma}
% % % % End Proposition% % % % % %

%The existence of positive entire solution of \eqref{2.2} is a standard diagonalization argument applying to an approximating sequence of  $L^2$ functions. 
To prove this, we will need an appropriate regularity estimate.
% for eigenfunctions in order to take a uniform limit from sequence.

% % % % % % % % % %
\begin{lemma}
\lb{L.2.2}
Assume that $\lambda>a_+$ and $\phi>0$ is continuous on $\bbR$ and solves \eqref{2.2} on $[b,\infty)$.  
There are $C=C(J,\la-a_-)>0$ and $m=m(J,\la-a_-,\la-a_+)>0$ such that the following hold.

(i)   If $x\ge b+\delta$, then
\begin{equation}
\lb{new-2.3}|(J*\phi )'(x)|\le C(J*\phi)(x).
\end{equation}
In particular, for all $x,y\in[b+\delta,\infty)$,
\begin{equation}
\lb{new-2.4}(J*\phi)(y)\le e^{C|x-y|}(J*\phi)(x)
\end{equation}

(ii) If $\lim_{x\to\infty} \phi(x)= 0$ and $y\ge x\ge b+\delta$, then 
\begin{equation}
\lb{new-2.5}(J*\phi)(y)\le \frac Cm e^{-m(y-x)} (J*\phi)(x).
\end{equation}
\end{lemma}
% % % end lemma % % %

{\it Remark.}  \eqref{new-2.4} and \eqref{2.2} imply $\phi(y)\le (1+\la -a_-) e^{C|x-y|} \phi (x)$ for $x,y\in[b+\delta,\infty)$ (note that $1+\lambda-a_+\ge 1$). This is a special case of the main result in \cite{CovHarnack}. 

\begin{proof}
(i) Let us  rewrite \eqref{2.2} for $x\ge b$ as
\begin{equation} \lb{2.6} 
(J*\phi)(x)=(1+\la-a(x))\phi(x).
\end{equation}
%From the expression, since $J\in C^1$ by \eqref{J1}, and $\la>a_+$, the derivative exists. 

Since $\supp\, J'\subseteq [-\delta,\delta]$ 
%$\sup\, J*J=[-2\de,2\de]$, 
and $J*J>0$ on $(-2\de,2\de)$, we have
%there is a constant, depending on $J$, such that for all $x$, $|J'(x)|\le C_JJ*J(x)$. $C_J$ can be chosen as 
\[
C_J:= \frac{\|J'\|_\infty}{\inf_{x\in[-\de,\de]}(J*J)(x)}>0
\]
and $|J'(x)|\le C_J(J*J)(x)$ for $x\in\bbR$.
Hence, by $\phi,J\ge 0$ and \eqref{2.6}, we have for $x\ge b+\delta$,
\[ |(J*\phi)'(x)|\le (|J'|*\phi)(x)
\le C_J (J*J*\phi)(x)
\le C_J(1+\la-a_-) (J* \phi)(x) 
\]
(we need $x\ge b+\delta$ in the last inequality). This is \eqref{new-2.3} when we take $C:=C_J(1+\la -a_-)$. 
 
(ii) We first claim that there is  $m=m(J,\la-a_-,\la-a_+)>0$ such that for $x\ge b+\delta$,
\begin{equation}
\lb{2.9}(J*\phi)(x)\ge m \int_x^\infty (J*\phi)(\tau)d\tau.
\end{equation}
Let us assume this is the case.   
%and define 
%\[ \Phi(x):= e^{mx}\int_x^\infty (J*\phi)(\tau)d\tau. \]
Then $[e^{mx}\int_x^\infty (J*\phi)(\tau)d\tau]'\le 0$, so for $y\ge x\ge b+\delta$, 
\begin{equation}
\lb{2.10}\int_y^\infty (J*\phi)(\tau)d\tau\le e^{-m(y-x)}\int_x^\infty (J*\phi)(\tau)d\tau.
\end{equation}
Hence, by this, \eqref{new-2.3}, and \eqref{2.9},
\begin{align*}
(J*\phi)(y)&=-\int_y^\infty (J*\phi )'(\tau)d\tau
\le C \int_y^\infty (J*\phi)(\tau)d\tau\\
&\quad \le Ce^{-m(y-x)}\int_x^\infty (J*\phi)(\tau)d\tau
\le \frac{C}{m}  e^{-m(y-x)} (J*\phi)(x).
\end{align*}
%Replacing $C$ by $C\max\{1,m^{-1}\}$ then proves the claim.

It remains to prove \eqref{2.9}.   We have $\lim_{x\to\infty} (J*\phi)(x)=0$ by the hypothesis.  Hence for each $\eps>0$, there is $R_0$ such that $\sup_{y>R_0}(J*\phi)(y)<\eps$. Let $R:=\max\{R_0,b+\delta,\tfrac 1\eps\}$. Then, 
\begin{equation}
\int_x^R (J*\phi)(\tau )d\tau =\int_x^{x+\de} (J*\phi)(\tau)d\tau +\int_{x+\de}^R (J*\phi)(\tau)d\tau
=I+II. \lb{2.11}
\end{equation}
By \eqref{new-2.4}, $I\le \de e^{C\de}(J*\phi)(x)$. On the other hand, by \eqref{2.6} and \eqref{J2},
\[
II\le \int_{x}^{R+\de } \phi(\tau) d\tau=\int_x^{R+\de} \frac{(J*\phi)(\tau)}{1+\la-a(\tau)}d\tau
\le \frac 1{1+\la-a_+} \int_{x}^R (J*\phi)(\tau)d\tau+  \frac {\eps \de}{1+\la-a_+}.
\]
The estimates for $I$ and $II$ and \eqref{2.11} now yield
\[ 
\de e^{C\de} (J*\phi)(x)\ge \frac{\la-a_+}{1+\la -a_+} \int_x^R (J*\phi)(\tau)d\tau-\frac{\eps\de}{1+\la-a_+}, 
\]
and \eqref{2.9} follows by letting $\eps\to 0$, with $m:=\delta^{-1}e^{-C\delta}(\lambda-a_+)(1+\lambda-a_+)^{-1}$.
\end{proof}

\begin{proof}[Proof of Lemma \ref{L.2.1}] 
%We begin by consider spectral radius of $H+a(x)$ as a bounded linear operator on $L^2(\mbr)\mapsto L^2(\mbr)$. Since $H$ is negative definite self adjoint operator (i.e, $\forall \psi\in L^2$, $\inn{H\psi,\psi}\le 0$; self adjoint is by \eqref{J1}), Rayleigh formula implies that the spectral radius is less than $a_+=\sup_x a(x)$. 
Obviously, $\la\notin \si(H+a(x))$ by $\la>a_+$. 
%$H+a(x)-\la$ is bijective on $L^2\mapsto L^2$. 
Let $0\not\equiv \eta\le 0$ be continuous and compactly supported and let $\varphi := (H+a(x)-\la)^{-1}\eta\in L^2(\bbR)$. Since
$J\in L^2(\bbR)$ as well, $J*\varphi$ is uniformly continuous and $\lim_{|x|\to\infty} (J*\varphi)(x)=0$.  Since also, 
\[ \varphi=\frac{J*\varphi-\eta}{1+\la-a}, \]
 $\varphi$ is continuous and $\lim_{|x|\to\infty} \varphi(x)=0$. 

Furthermore, $\varphi>0$. Indeed, otherwise $\varphi$ achieves a non-positive minimum, and since $\varphi\not\equiv 0$, the set of global minima of $\varphi$ has a boundary point $x_0$.  
%there is $x_0\in\mbr$ such that $\varphi(x_0)=\min_{x\in \mbr}\varphi(x)\le 0$, we obtain 
From the properties of $J$ now follows that $(H\varphi)(x_0)> 0$.   But then $\eta(x_0)=(H\varphi)(x_0)+(a(x_0)-\la)\varphi(x_0)> 0$ by $\la>a_+$, contradicting $\eta\le 0$.  Thus $\varphi>0$, and Lemma \ref{L.2.2} applies to $\varphi$.

Let us choose $\eta$ with $\supp\, \eta=[-1,0]$, define $\eta_j(x):=\eta(x+j)$,  $\varphi_j:= (H+a(x)-\la)^{-1}\eta_j$, and $\phi_j:=\varphi(0)^{-1} \varphi_j$ (recall that $\varphi(0)>0$).  Then $\phi_j$ solves \eqref{2.6} on $[-j,\infty)$, so Lemma \ref{L.2.2}(i) gives
for $x\ge -j+\delta$, %$x\ge -j+\delta$,   
\[
\left| \left[ \log \left ([1+\la -a(x)]\phi_j(x) \right) \right]' \right|= \left| \left[ \log (J*\phi_j) \right]'(x) \right|\le C.
\]
Since also $\log \phi_j(0)=0$, 
%and $\sup_x|\log(1+\la -a(x))|<\infty$, 
there is a locally uniform limit $\Phi>0$ for some subsequence of \{$(1+\la -a(x))\phi_j(x)\}_j$. Let $\phi(x) := (1+\la -a(x))^{-1}\Phi(x)$, which is positive and continuous. We have $\phi_j\to \phi$ locally uniformly because $1+\lambda-a_+\ge 1$, 
%By \eqref{new-2.4}, $\phi_j(y)\le (J*\phi_j)(y)\le e^{C|y|}(J*\phi_j)(0)\le  e^{C|y|}(1+\la-a_-)$ for $y\ge -j+\de$. So, $(J*\phi_j)(y)\to (J*\phi)(y)$ by dominated convergence theorem, and $\phi$ solves \eqref{2.2} on $\mbr$.
hence $\phi$ solves \eqref{2.2} on $\mbr$. 

By Lemma \ref{L.2.2}(ii) and \eqref{2.6}, for any $j$ and $y\ge x\ge -j+\delta$ %$y\ge x\ge -j+\delta$ 
we have 
\[ \phi_j(y)\le (J*\phi_j)(y) \le  \frac{C}{m} (1+\la -a_-)e^{-m(y-x)}\phi_j(x).\]
Hence Lemma \ref{L.2.1} holds with $\phi_\lambda:=\phi$ and $L:=\max\{\tfrac 1m\log \frac{2C(1+\lambda-a_-)}m,\delta\}$. 
%for $y\ge 0$ we have $\phi(y)\le Ae^{-my}$. 
%Passing the inequality to its limit, $\phi(y)\le Ae^{-my}$. So $\phi(x)\to 0$ as $x\to\infty$. 
%Finally, applying Lemma \ref{L.2.2} to $\phi$, with $x\le y=0$, yields $A^{-1}e^{-mx}\le \phi(x)$.  So $\phi_\lambda:=\phi$ satisfies the proposition and the proof is finished.
\end{proof}

%$$$$$$$$$$$$$$$$$$$$$$$$$$$$$$$$$$$
\section{Proof of Theorem \ref{T.1.1}(ii) (Construction of a Sub-Solution)} \lb{S3}
%$$$$$$$$$$$$$$$$$$$$$$$$$$$$$$$$$$$

We now turn to the construction of sub-solutions of \eqref{1.1}, extending the method from \cite{ZlaInhomog}.  
%The idea is to find a sub-solution of the form $w=h(v)$, with $v$ the super-solution from \eqref{1.5} constructed in the previous section and an appropriate function $h:[0,\infty)\mapsto [0,1]$.  
The function $h_g$ will be taken from a family of functions $\{h_{g,\al}\}_{\al\in(0,1)}$ satisfying \eqref{1.32}, which have been constructed in \cite{ZlaInhomog} (we note that our $h_{g,\al}$ equals $h_{g,\al^2}$ from \cite{ZlaInhomog}).

It was proved in \cite{Uchi}  that under the hypotheses (G) and for each $\al\in(0,1)$, the homogeneous PDE $u_t=u_{xx}+g(u)$ (with classical diffusion) has a (unique) {\it traveling front} solution $u(x,t)=U_{g,\alpha}(x-c_\alpha t)\in(0,1)$ (with $c_\alpha:=\al+\al^{-1}$) which satisfies $\lim_{s\to\infty} e^{\alpha s}U_{g,\alpha}(s)=1$.  The pair $(U_{g,\al},c_\al)$ here solves the traveling front boundary value problem
\begin{equation} \lb{3.2}
U_{g,\al}''+c_\al U_{g,\al}'+g(U_{g,\al})=0, \qquad  \lim_{s\to-\infty} U_{g,\al}(s)=1, \qquad \lim_{s\to\infty} U_{g,\al}(s)=0, 
\end{equation}
whose solutions  are (up to translation in $s$) precisely $\{(U_{g,\al},c_\al)\}_{\al\in(0,1]}$.  They satisfy $U_{g,\al}'<0$ on $\bbR$, and the critical front $U_{g,1}$ (which we will not use) satisfies  $\lim_{s\to\infty} s^{-1}e^{ s}U_{g,1}(s)=1$.

The linearization $v_t=v_{xx}+v$ of $u_t=u_{xx}+g(u)$ at $u=0$ has corresponding traveling front solutions $v(x,t)=e^{-\alpha (x-c_\al t)}$, and $h_{g,\al}$ is chosen to be the function which takes $e^{-\al s}$ to $U_{g,\al}(s)$ for $\al\in(0,1)$.  That is,  
\begin{equation}
\lb{3.1}
h_{g,\al}(v):=\left\{ 
\begin{array}{ll}
U_{g,\al}(-\al^{-1}\log v) & v>0, \\ 
0 & v=0.
\end{array} \right. 
\end{equation}
Notice that \eqref{3.2} yields
\begin{equation} \lb{3.3} 
\al^2 v^2 h''_{g,\al}(v)-vh'_{g,\al}(v)+g(h_{g,\al}(v))=0,
\end{equation} 
and \eqref{1.32} follows from the definition of $h_{g,\al}$, with $h_{g,\al}'(0)=1$ due to $\lim_{s\to\infty} e^{\alpha s}U_{g,\alpha}(s)=1$, and $h_{g,\al}''<0$ proved in \cite{ZlaInhomog} (also in Lemma \ref{L.5.1} below). 

It turns out that the same $h_{g,\al}$ can be used for our non-local diffusion problem \eqref{1.1}.  To do that, we will need the following two lemmas, whose proofs we postpone until after the proof of Theorem \ref{T.1.1}.

\begin{lemma} \lb{L.3.1} 
Let $g$ satisfy (G) and for $\al\in(0,1)$ let $\be:=2+\al^{-2}$ and $h_{g,\al}$ be from \eqref{3.1}.  Then $\rho_{g,\al}(x):=-h''_{g,\al}(e^{-x})>0$ satisfies $|\rho_{g,\al}'(x)|\le \be\rho_{g,\al}(x)$ for $x\in\bbR$ and, in particular, $\rho_{g,\al}(y)\le e^{\be|x-y|}\rho_{g,\al}(x)$ for $x,y\in\bbR$.
\end{lemma}

\begin{lemma} \lb{L.3.2} 
Let $\phi_\la>0$ satisfy \eqref{2.2} with $\la>a_+$ and \eqref{1.6}. For each $s>0$ there is $\gamma_s=\ga_s (J)>0$ with $\lim_{s\searrow 0}\gamma_s= 0$ and such that if $|x-y|\le \de$ (with $\de$ from \eqref{J2}), then 
\begin{equation}
|\phi_\la(x)-\phi_\la(y)|\le \gamma_{\la-a_-}\phi_\la(y).
\end{equation}
%Moreover,  $\lim_{s\to 0}\gamma_s= 0$. 
\end{lemma}

{\it Remark.}
 Lemma \ref{L.3.2} is an improvement of the remark after  Lemma \ref{L.2.2}. 
 %We do not know whether \eqref{2.3} can also be improved accordingly.
 \smallskip

Let  $h_g :=h_{g,3/4}$, with $h_{g,\al}$ from \eqref{3.1}.
%, where $\alpha\in(\tfrac 12,1)$ will be chosen later (it will depend on $\la$).  
We will suppress the subscripts $g,\la$ in what follows, denoting $w=w_\la=h_g(v_\la)=h(v)$. Then by \eqref{1.5} and \eqref{2.2}, 
\[
w_t-Hw = h'(v)Hv+a(x)vh'(v)-\int_{-\delta}^{\de} J(y)[w(x-y,t)-w(x,t)]dy. \lb{3.5}
\]
By Taylor's theorem for $h(v)$ we have 
\[
w(x-y,t)-w(x,t) = h'(v(x,t))[v(x-y,t)-v(x,t)]+\frac 12 h''(\zeta_{x,y,t})[v(x-y,t)-v(x,t)]^2,
\]
where $\zeta_{x,y,t}$ is some number between $v(x-y,t)$ and $v(x,t)$. This and the definition of $Hv$ yield
\beq \lb{3.6}
w_t-Hw=a(x)vh'(v)-\frac 12 \int_{-\de}^{\de} h''(\zeta_{x,y,t})J(y)[v(x-y,t)-v(x,t)]^2 dy.
\eeq
Since $\zeta_{x,y,t}$ is between $v(x-y,t)$ and $v(x,t)$ (and $|y|\le\del$), Lemma \ref{L.3.2} implies 
%\[
%(1+\gamma_\la)^{-1}v(x,t)\le \zeta_{x,y,t}\le (1+\gamma_\la)v(x,t),
%\]
\[
|\log \zeta_{x,y,t}-\log v(x,t)|\le \log (1+\gamma_{\la-a_-}).
\]
 Lemma \ref{L.3.1} with $\beta=2+(3/4)^{-2}<4$ now gives
\beq \lb{3.7}
-h''(\zeta_{x,y,t})=\rho(-\log \zeta_{x,y,t})
\le e^{4 \log (1+\gamma_{\la-a_-})}\rho(-\log v(x,t))
= -(1+\gamma_{\la-a_-})^4 h''(v(x,t)).
\eeq
On the other hand, by Lemma \ref{L.3.2}, 
\begin{equation}
\lb{3.8}\int_{-\de}^{\de} J(y)[v(x-y,t)-v(x,t)]^2dy\le \gamma_{\la-a_-}^2v(x,t).
\end{equation}
Using \eqref{3.7}, \eqref{3.8}, and $h''<0$, we obtain from \eqref{3.6},
\begin{equation}
\lb{3.9}w_t-Hw\le a(x) vh'(v)-\frac 12 \gamma_{\la-a_-}^2(1+\gamma_{\la-a_-})^4 v^2 h''(v).
\end{equation}

Since $\lim_{s\searrow 0}\gamma_s= 0$ by Lemma \ref{L.3.2},  there exists (non-decreasing in $a_-$) $\la_0=\la_0(J,a_-)$ such that 
\[ \gamma_{s}^2(1+\gamma_{s})^4 \le   a_-  \]
for all $s\in(0,\la_0)$.  If now $a_+<a_-+\la_0$ and $\la\in(a_+,a_-+\la_0)$, then  we have 
\[\frac 12 \gamma_{\la-a_-}^2(1+\gamma_{\la-a_-})^4\le  \left(\frac 34 \right)^2 a(x).  \]
%(recall that $\beta=2+\alpha^{-2}$). \comment{Is the following a better way to state this? See code.}
%There is $\la_0>0$, if $\la \in (a_-,a_-+\la_0)$, there is $\al \in (0,1)$, such that 
%\[ \frac 12 \ga _\la ^2 (1+\ga _\la )^{\be (\al )}\le \al ^2 a_-\le \al ^2 a(x). \]
Thus \eqref{3.9}, $h''<0$,  \eqref{3.3} for $h=h_g=h_{g,3/4}$, and \eqref{F2}  yield %for such $\alpha$
\[
w_t-Hw\le a(x)[vh'(v)-(3/4)^2 v^2 h''(v)]
=a(x)g(w)
\le f(x,w).
\]
So  $w=w_\la=h_{g}(v_\lambda)$ is a sub-solution of \eqref{1.1}.

%$$$$$$$$$$$$$$$$$$$$$$$$$$$$$$$$$$$$
\section{Proof of Theorem \ref{T.1.1}(iii) (Construction of a Transition Front)}  \lb{S4}
%$$$$$$$$$$$$$$$$$$$$$$$$$$$$$$$$$$$$

For reaction-diffusion equations with classical diffusion, there is a simple and standard way to construct a transition front for \eqref{1.1} between the  super-solution $v_\la$ and sub-solution $w_\la=h_g(v_\la)\le v_\la$ from the last two sections.  One lets  $u_n:\bbR\times (-n,\infty)\to[0,1]$ be the solution of the Cauchy problem with initial datum $u_n(x,-n)$ between $w_\la(x,-n)$ and $\min\{v_\la(x,-n),1\}$, and recovers a transition front $u_\la:\bbR^2\to[0,1]$ as a locally uniform limit along a subsequence of $\{u_n\}_{n\ge 1}$, using parabolic regularity results and the Arzel\` a-Ascoli theorem.  

Such regularization results are not available for the non-local diffusion operator $H$, as was discussed in the introduction. 
Nevertheless, $H$ does not (qualitatively) worsen the regularity of the solutions of \eqref{1.1}, so one might hope that if the initial datum is  sufficiently regular (in our case, Lipschitz or H\" older continuous would suffice) then this regularity will persist indefinitely for bounded  solutions.  In fact, a simple argument from \cite{LiSunWang} (where the homogeneous case was treated) shows that if $\sup_{(x,u)\in\bbR\times[0,1]} f_u(x,u)< 1$, then Lipschitz initial data give rise to uniformly-in-time (and $n$) Lipschitz solutions.  We do not assume such a bound here, and thus will have to prove a similar result for the sequence of solutions $u_n$ in a different way.

%Having proved \eqref{new-2.3} and \eqref{1.32}, we see that $w_\la$ from the last section is a Lipschitz function, which  suggests to take $u_n(x,-n):=w_\la(x,-n)$.  
We consider the Cauchy problem 
\begin{equation}\lb{4.1}
\left\{
\begin{array}{ll}
u_t=Hu+f(x,u) & \text{on $\mbr\times(-n,\infty)$,} \\ 
u(x,-n)=w(x,-n) \quad(=h(v(x,-n))) & \text{on $\bbR$,}
\end{array} 
\right .
\end{equation}
where we dropped the subscripts $n,g,\la$.  
The proof of existence and uniqueness of a bounded continuous classical solution to this problem with bounded continuous initial data is standard, and identical to the homogeneous case (see, e.g., \cite{LiSunWang}).  The proofs of the maximum and comparison principles for \eqref{1.1} are also standard.  These imply, in particular, 
\beq\lb{4.2z}
w\le u\le \min\{v,1\}.
\eeq

%Here, $w$ is the sub-solution we obtain in previous section. 
We then obtain the following bound on $u$ from \eqref{4.1}.

\begin{lemma}\lb{L.4.1}
There is $\bar C=\bar C(J,\la-a_-,||f||_{C^2},g)$  such that the solution of \eqref{4.1} satisfies
\begin{equation}\lb{4.2}
\eta(t):=\sup_{0<|y-x|\le\del} \frac{|u(y,t)-u(x,t)|}{|y-x|u(x,t)}\le \bar  C
\end{equation}
for any $t\ge-n$.
% and $s\in[-\del,\del]\setminus\{0\}$.
\end{lemma}

{\it Remark.}  In particular, $u_x(\cdot,t)$ exists almost everywhere for each $t\ge -n$, and $|u_x|\le \bar Cu$.

\begin{proof}
From \eqref{new-2.3}, \eqref{2.6}, $\|a\|_{C^1}\le \|f\|_{C^2}$, and $1+\lambda-a_+\ge 1$ we have  %(with $\phi:=\phi_\lambda$)
\begin{equation}\lb{new-4.3-1}
|\phi_\lambda '(x)| = \left| \frac {(J*\phi_\lambda)'(x) + a'(x)\phi_\lambda(x)} {1+\lambda-a(x)}  \right| \le \left[ C(1+\la -a_-)+||f||_{C^2} \right]\phi_\lambda(x)=:C_1\phi_\lambda(x).
\end{equation}
Since $v(x,t)=e^{\la t} \phi_\la(x)$, $|v_x|\le C_1v$. From concavity of $h$ we have $vh'(v)\le h(v)$.  Thus 
\[ |w_x|=h'(v)|v_x|\le C_1h'(v)v\le C_1h(v)=C_1w, \]
so $\eta(-n)\le C_1e^{C_1\de }$ (a bound which is independent of $n$).

%Secondly, $u$ has Harnack property. 
The comparison principle for \eqref{1.1} shows $w\le u\le \til v:=\min\{v,1\}$ on $\mbr\times[-n,\infty)$. Concavity of $h$ then yields $\tilde v\le h(v)h(1)^{-1}=wh(1)^{-1}\le u h(1)^{-1}$.  Since from \eqref{new-4.3-1} we have $\tilde v(x,t)\le e^{2C_1\delta} \tilde v(y,t)$ for $|x-y|\le 2\de$ (with $\de$ from \eqref{J2}), it follows that
\begin{equation}\lb{4.3}
u(y,t)\le \til C u(x,t)
\end{equation}
for $|x-y|\le 2\de$ and $\til C:= e^{2C_1\del}h(1)^{-1}$.

Let now $u^s(x,t):=u(x+s,t)$, $q^s:=\tfrac 1s (u^s-u)$, and $z^s:= q^s/u$.  The lemma will follow if we show $|z^s(x,t)|\le \bar C$ for some $\bar C=\bar C(J,\la-a_-,||f||_{C^2},h)<\infty$ and all $x\in\bbR$, $t\ge-n$, and $0<|s|\le\del$ (recall that $h=h_g$ only depends on $g$).  We have
\begin{equation}
q^s_t-Hq^s =\frac{f(x+s,u^s)- f(x,u^s)}{s}+ \frac{f(x,u^s)-f(x,u)}{s}.  \lb{4.4}
\end{equation}
%, and the absolute value of the right-hand side is bounded by $\|f\|_{C^1}(1+|q^s|)$.  The comparison principle now shows that
%\[
%\|q^s(\cdot, t)\|_\infty \le [1+\eta(-n)]e^{\|f\|_{C^1} t}-1
%\]
%for each $s\neq 0$, so $\eta(t)\le [1+\eta(-n)]e^{\|f\|_{C^1} t}-1$.  Hence $u(\cdot, t)$ is Lipschitz for each $t\ge 0$.
%We will now use \eqref{4.3} to show that in fact $\sup_{t\ge-n}\eta(t)\le C$, for some $C=C(J,f,\la,h)$.  
By \eqref{4.1} and \eqref{4.4}, 
\beq\lb{4.3z}
z^s_t= \al(x,t)+\be(x,t)z^s,
\eeq
with
\begin{eqnarray}%\lb{4.5a}
\al (x,t)&=&\frac{J*q^s}{u}+\frac{ f(x+s,u^s)-f(x,u^s)}{su},\lb{4.4a}\\
\be(x,t)&=&-\frac{J*u}{u}+\frac{f(x,u^s)- f(x,u)}{u^s-u}-\frac{f(x,u)}{u}.\lb{4.4b}
\end{eqnarray}

Recall $0<|s|\le\delta$. 
We have $J*q^s=\tfrac 1s(J^{-s}-J)*u$, %where $D^{-s}J=\frac 1 s (J(x-s)-J(x))$. 
so \eqref{4.3} implies $|J*q^s|\le 3\del ||J'||_{\infty} \tilde Cu$. Since also $f_x(\cdot,0)\equiv 0$, we obtain $|f(x+s,u^s)-f(x,u^s)|\le ||f||_{C^2}|s|u^s$, and \eqref{4.3} now gives %\le \til C||f||_{C^2}su$
\begin{equation}\lb{4.5}
|\al(x,t)|\le \tilde C \left( 3\del ||J'||_{\infty}+||f||_{C^2} \right)=: M.
\end{equation}
From \eqref{4.3} we obtain
\begin{equation}\lb{4.6}
- \frac{J*u}{u}  \le -\frac 1{\tilde C},
\end{equation}
as well as  
\begin{equation}\lb{4.7a}
\abs{\frac{f(x,u^s)- f(x,u)}{u^s-u}-\frac{f(x,u)}u}\le \frac 1{2\tilde C}
\end{equation} 
whenever $u\le \theta_0:=(2\til C^2 ||f||_{C^2})^{-1}$ (then also $u^s\le (2\til C ||f||_{C^2})^{-1}$).  Thus
\begin{equation}\lb{4.7}
\beta(x,t)\le -\frac 1{2\tilde C}
\end{equation}
when $u\le\tht_0$.

We now fix any $x\in\bbR$ and regard \eqref{4.3z} as an ODE in $t$.
If $t_x:=\inf\{t\ge -n:u(x,t)>\ta_0\}$, then \eqref{4.7} holds for all $t\in(-n,t_x)$.
%$v(x,t_x)\ge u(x,t_x)=\ta_0$. Because $h$ is increasing, $w(x,t_x)=h(v(x,t_x))\ge h(\theta_0)$. D
Next define, with $\tht_1$ from \eqref{F3},
\[ T:=\frac 1{\la} \log \frac{\til C h^{-1}(\ta_1)}{\ta_0}.\]
%, there is also $\theta_1\in(0,1)$ such that $f$ is decreasing in $u$ for $u\ge \theta_1$. 
From \eqref{4.2z}, $h'>0$, $u(x,t_x)\ge \ta_0$, and \eqref{4.3} we obtain for $|r|\le\del$ and $t\ge t_x+T$, 
\[ u^r(x,t)\ge h(v^r(x,t))\ge h(e^{\la T}v^r(x,t_x)) \ge h(e^{\la T}u^r(x,t_x)) \ge h(e^{\la T}\til C^{-1}\ta_0)=h(h^{-1}(\ta _1))= \ta _1. \]
So \eqref{F3} implies  
\[ \frac{f(x,u^s)- f(x,u)}{u^s-u}\le 0 \]
for $t\ge t_x+T$, and then \eqref{4.4b} and \eqref{4.6} show \eqref{4.7}
%\begin{equation}\lb{4.8}
%\be(x,t)\le -\frac 1 {\tilde C}\quad t\ge t_x+T
%\end{equation}
for $t\ge t_x+T$.
Finally, for $t\in[t_x,t_x+T)$, 
\begin{equation}\lb{4.9}
\be(x,t)\le ||f||_{C^1}.
\end{equation}

From \eqref{4.5} and \eqref{4.7} for $t\in(-n,t_x)$ we obtain
% By comparison principle for ODE, $z$, the solution of $z_t=\al+\be z$, is bounded by a constant depends on its initial data $z(x,-n)$, $\tilde C$, $T$, $M$ (from \eqref{4.5}), and $||f||_{C^1}$ all the time $t\ge -n$, where initial data is shown to be bounded $|z(x,-n)|\le \tilde C$ in the beginning of the proof. So $z$ is bounded by a constant not depend on $x$ and $t$. In particular, the comparison gives the following bound:
$z(x,t)\le \til \max\{\eta(-n),2\tilde CM\}$ for $t\le t_x$ (recall that $\eta(-n)$ is bounded uniformly in $n$), and then  \eqref{4.9} for $t\in[t_x,t_x+T)$ and \eqref{4.7} for $[t_x+T,\infty)$ yield
\[ |z^s(x,t)|\le \rb{\max\{\eta(-n),2\tilde CM\}+\frac{M}{||f||_{C^1}}}e^{||f||_{C^1}T}-\frac{ M}{||f||_{C^1}}=:\bar C \]
for all $t\ge -n$, and $x\in\bbR$ and $0<|s|\le\del$. This proves  \eqref{4.2}. 
\end{proof}

{\it Remark.}  The Harnack-type bound \eqref{4.3} played a crucial role in the above proof.  We note that without it, one can still prove that $\eta(t)$ is locally bounded if it is finite initially.  Indeed, the absolute value of the right-hand side of \eqref{4.4}  is bounded by $\|f\|_{C^1}(1+|q^s|)$, so the comparison principle shows (with initial time $t_0$)
\[
\|q^s(\cdot, t)\|_\infty \le [1+\eta(t_0)]e^{\|f\|_{C^1} (t-t_0)}-1
\]
for each $s\neq 0$.  Hence, $\eta(t)$ satisfies the same bound.
%\le [1+\eta(t_0)]e^{\|f\|_{C^1} (t-t_0)}-1$.
\smallskip

%***************Alternative paragraph********
%\comment{TS--}In \cite{LiSunWang}, this issue is resolved by introducing additional restriction on regularity of $f$ (in particular, $||f_u||<1$, see \cite{LiSunWang}, proposition 2.5). In the present paper, we prove global Lipschitzness of solution under assumption (J) and (F). The proof relies on derivative estimate in \eqref{2.4}. \comment{--TS}
%
%\rem The proof relies on Harnack property of $u$. In fact, parabolic regularity almost never true for nonlocal case. Suppose there is some point $(x,t)$ such that $u\approx 0$ and $u\ll u_x$. Then \eqref{4.4} gives that $u_{xt}\ge -C+(f_u(x,u)-1)u_x$ for some constant $C$. So $u_x$ is increasing drastically for short time $t$ so long as $f_u-1\gg 0$. (??)
%**********************************************

%\begin{proof}[Proof of Theorem \ref{T.1.1}, cont'd]
Let $u_n$ be the (unique) solution of \eqref{4.1}. The constant $\bar C$ from Lemma \ref{L.4.1} is a uniform-in-$n$ bound on $|(u_n)_x|$ because $0\le u_n\le 1$.  Since $Hu+f(x,u)$ is also uniformly bounded in $0\le u\le 1$, we find that $|(u_n)_t|\le 2+\|f\|_{C^1}$.
% is also uniformly bounded by some $C'$.  
Since  
\[
\frac\partial{\partial t}[Hu+f(x,u)]= Hu_t+f_u(x,u)u_t
\]
by the dominated convergence theorem, we have $|(u_n)_{tt}|\le (2+\|f\|_{C^1})^2$.  Thus we see that $u_n$ and $(u_n)_t$ converge, along a  subsequence,  locally uniformly to $u_\lambda$ and $(u_\lambda)_t$ for some $u_\lambda:\bbR^2\to[0,1]$. Then obviously $u$ solves \eqref{1.1}, and \eqref{1.7} holds by \eqref{4.2z} for each $u_n$. 

%(3) Let $u_n$ be the (unique) solution of \eqref{4.1}. By lemma \ref{L.4.1}, because $u_n\le 1$, $(u_n)_x$ is bounded by $C$. From \eqref{4.4}, this implies $|(u_n)_{xt}|\le C'$. By \eqref{4.1}, it is obvious that $|(u_n)_t|\le C''$ for some constant $C''$ by $u_n\le 1$. Taking $t$-derivative for \eqref{4.1}, it shows that $|(u_n)_{tt}|\le C'''$ for some constant $C'''$. Picking a subsequence of necessary, we may assume that $u_n\to u$ and $(u_n)_t\to u_t$ uniformly for some $u:\mbr\times\mbr\mapsto [0,1]$. Immediately, $u$ solves \eqref{1.1}, and \eqref{1.7} holds. 

From \eqref{1.7} we obtain \eqref{1.3}, so it remains show \eqref{1.4}.  If $L$ is from Lemma \ref{L.2.1} for $\phi_\la$ from \eqref{1.5}, then the lemma and \eqref{1.7} yield
\[
\sup_{t\in\bbR} L_{u,\eps}(t) \le L \lceil \log_2( \eps^{-1}h_g^{-1}(1-\eps)) \rceil,
\]
which gives \eqref{1.4}. So $u$ is a transition front and the proof of Theorem \ref{T.1.1} is finished. 

%  shows that for any $B>1$, there is $L>0$ such that for any $x_0\in \mbr $, if $x\le x_0-L$, then $\phi_\la (x)\ge B\phi_\la (x_0)$. $L$ can be chosen as $L:=m^{-1}\log (AB)$. Now given $\ep \in (0,1/2)$, denote 
%\[ L_{u,\ep}(t) :=\diam \{x\in \mbr : \ep \le u(x,t)\le 1-\ep \}.  \]
%\eqref{1.4} follows if $L_{u,\ep}(t)\le L_\ep $ uniformly in $t$ for some $L_\ep $. Because of $w\le u\le v$, $w=h(v)$, and \eqref{1.5}, we have 
%\begin{equation}\lb{4.10}
%\{x\in \mbr: \ep \le u(x,t)\le 1-\ep \}\subset \{ x\in \mbr :e^{-\la t}\ep\le \phi _\la (x)\le e^{-\la t}h^{-1}(1-\ep )\}:=Z_{\ep}(t).
%\end{equation}
%Denote $B=B_\ep := h^{-1}(1-\ep)/\ep $. Let $L=L_\ep $ as before. Then, \eqref{4.10} leads to
%\[ L_{\ep ,u}(t)\le \diam Z_\ep (t) \le L_\ep . \]

%\end{proof}

%$$$$$$$$$$$$$$$$$$$$$$$$$$$$$$$$$$$$$$$$$$$$$$$$$$$$$
\section{Proof of Lemma \ref{L.3.1} (Estimate on the Third Derivative of $h_{g,\al}$)}
%$$$$$$$$$$$$$$$$$$$$$$$$$$$$$$$$$$$$$$$$$$$$$$$$$$$$$

We will again drop the subscript $g,\al$ in 
%Also, for convenience, we will prove the same result (lemma \ref{L.3.1}) for $\rho(x)=-h''(e^{-x})$ instead of $-h''(e^{x})$. This is to avoid the negative sign confusion later. 
$\rho_{g,\al}$, $h_{g,\al}$, and $U_{g,\al}$.  From \eqref{3.2}, \eqref{3.1}, and $c_\al=\al+\al^{-1}$ we have
\begin{equation}\lb{5.1}
\rho(x)=\al^{-3} e^{2x} \left[U'(\al^{-1}x)+\al g(U(\al^{-1}x))\right]= \al^{-3} e^{2x} \eta(\al^{-1}x),
\end{equation}
with $\eta=\eta_{g,\al}$ given by 
\begin{equation}\lb{5.2}
\eta:= U'+\al g(U).
\end{equation}
By differentiating we obtain
\begin{equation}\lb{5.3}
\rho '(x)=2\rho(x)+\al^{-4} e^{2x} \eta '(\al^{-1} x).
\end{equation}
Thus Lemma \ref{L.3.1} will follow if we  show $|\eta '|\le \al^{-1} \eta$.
% ($\eta>0$ because of lemma \ref{5.1}). 
Using \eqref{3.2} and $c_\al=\al+\al^{-1}$, we obtain
\begin{equation}\lb{5.4}
\eta '= -\al^{-1}\eta-\al U' (1-g'(U)).
\end{equation}
Since $U'<0\le 1-g'(U)$, the latter by \eqref{G2}, it suffices to prove $-\al U' (1-g'(U))\le 2\al^{-1} \eta$.  By \eqref{5.2}, this is equivalent to 
\beq\lb{5.5}
-U'\le \frac{2\al}{2+\al^2(1-g'(U))} g(U).
\eeq
Since $0\le 1-g'(U)\le 2$ by  \eqref{G2},  this (and hence Lemma \ref{L.3.1}) will be proved once we prove the following lemma.

\begin{lemma}\lb{L.5.1}
%Suppose $0<\al<1$, and $U=U_{g,\al}$ as before, and $g$ satisfies {\rm (G)}. 
For $r_{\al}:(-\infty,1]\to \mbr$, given by 
\begin{equation}
r_\al (v):= \left\{
\begin{array}{ll}
\cfrac{\al}{1+\al^2(1-v)} & v\in[0,1], \\ 
\cfrac{\al}{1+\al^2} & v<0,
\end{array} \notag
\right.
\end{equation}
we have $-U'\le r_\al(g'(U)) g(U)$. 
\end{lemma}

{\it Remark.} This is an improvement of Lemma 3.1 in \cite{ZlaInhomog}, which shows that $-U'\le \al g(U)$
 (and thus $\eta>0$ and $h''<0$).
\smallskip

%Before proving Lemma \ref{L.5.1}, let us show how it yields Lemma \ref{L.3.1}.
%%By the calculation we did at the beginning of section, it suffices for us to prove $(1-g'(U))U'\lesssim \eta$, where $\eta$ as in \eqref{5.2}. Again, we split into two cases: $x>x_0$ and $x\le x_0$, where $x_0$ is the unique number such that $g'(U(x_0))=0$. 
%Let $x_0\in\bbR$ be the unique number such that $g'(U(x_0))=0$.  If $x>x_0$, then $0\le g'(U(x))\le 1$, so by Lemma \ref{L.5.1},  
%\[ -U'(1-g'(U))\le \frac 1{\al^2}\eta. \]
%For the later case $x\le x_0$, $g'(U(x))\le 0$. By lemma \ref{L.5.1} and condition \eqref{G2}, 
%\begin{eqnarray*}
%-(1-g'(U))U'&\le& -2 U'\\
%&=&2 \cb{\frac 1{\al^2} U'-\rb{1+\frac 1{\al^2}}U'} \\
%&=&\frac{2}{\al^2}\eta.
%\end{eqnarray*}
%Both cases conclude that 
%\begin{equation}\lb{5.10}
%-(1-g'(U))U'\le \frac 2{\al^2}\eta.
%\end{equation}
%So \eqref{5.4} gives 
%\begin{equation}\lb{5.11}
%|\eta'|\le \frac 1{\al} \eta.
%\end{equation}
%Return to \eqref{5.3}. \eqref{5.1} and \eqref{5.11} imply
%\begin{eqnarray*}
%|\rho '(x)|&\le &2\rho(x)+\frac 1{\al^4} e^{2x} \eta'(\al^{-1}x)\\
%&=&\rb{2+\frac 1\al}\rho(x).
%\end{eqnarray*}
%This is exactly lemma \ref{L.3.1}. 

\begin{proof}[Proof of Lemma \ref{L.5.1}]
We will in fact prove the stronger estimate $-U'\le q(g'(U))g(U)$, where $q:(-\infty,1]\to \mbr$ is given by (recall that $c_\al=\al+\al^{-1}\ge 2$)
\begin{equation}
q (v)\equiv \left\{
\begin{array}{ll}
\cfrac{2}{c_\al+\sqrt{c_\al^2-4v}} & v\in[0,1], \\ 
\cfrac{1}{c_\al} & v<0.
\end{array} 
\right.
\end{equation}

It is easy to check that $q\le r_\al$ on $(-\infty,1]$. 
%Firstly, $q_\al$ is well-defined on $0\le x\le 1$. Particularly, $c_\al^2-4v\ge 0$ because $c_\al=\al+\al^{-1}\ge 2$ ($\al\in(0,1)$), and $v\le 1$. Also, 
Also, $q>0$ is continuous and non-decreasing, and for $v\in[0,1]$  we have 
\begin{equation}\lb{5.7}
vq(v)^2-c_\al q(v)+1=0.
\end{equation}

Since $g'$ and $U$ are decreasing, $g'(U(x))$ is increasing in $x$ with limits $g'(1)<0$ and $g'(0)=1$ as $x\to\pm\infty$. Let $x_0\in\mbr$ be the unique number such that $g'(U(x_0))=0$, and let us prove 
\beq\lb{5.1z}
-U'(x)\le q(g'(U(x)))g(U(x))
\eeq
separately for $x\ge x_0$ and $x< x_0$. 

First fix any $x\ge x_0$.  Then $g'(U(x))\in[0, 1]$. \eqref{5.7} shows that $s:= q(g'(U(x)))$ satisfies 
\begin{equation}\lb{5.8}
g'(U(x))s^2-c_\al s+1=0.
\end{equation}
Define the region $D_x\subseteq \mbr^2$ by 
\[ D_x:= \left\{(u,v): u\in(U(x),1) \text{ and } v\in(-sg(u),0) \right\}. \]
%$$$$$$$$$$figure$$$$$$$$$$$$$$$$
\begin{figure}[H]
\centering
\includegraphics[scale=1]{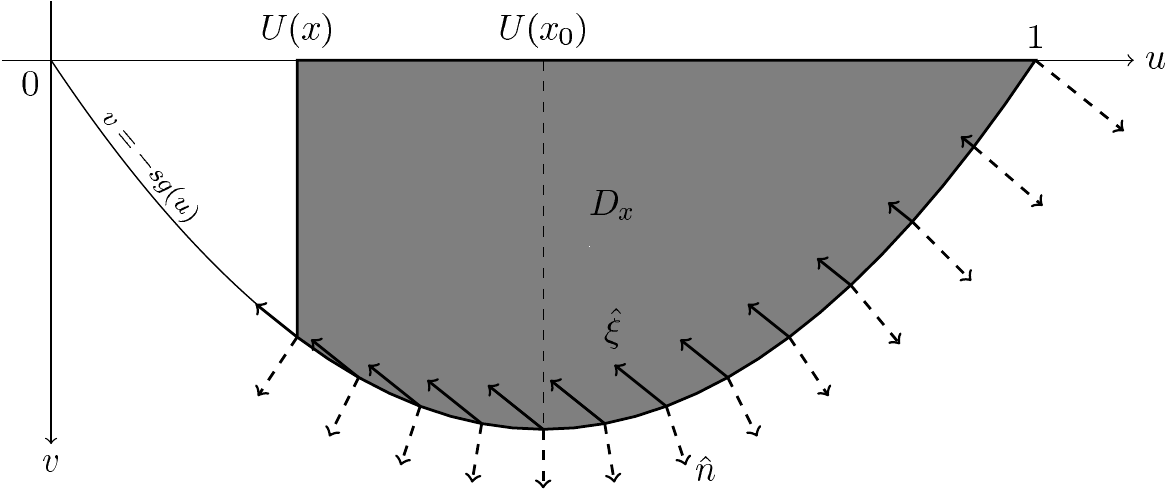}
%\begin{tikzpicture}[domain=0:1,xscale=10,yscale=30]
%\tikzset{declare function={g(\t)={-1/2*\t*(1-\t)};}};
%\tikzset{declare function={G(\t)={\t*(1-\t)};}};
%\fill[fill=gray] (0.25,0) -- plot[domain=0.25:1](\x,{g(\x})  -- cycle;
%\node[below] at (0.57,-0.04) {$D_x$};
%\tikzset{declare function={dG(\t)={1-2*\t};}}
%\draw[<->] (0,-0.13)node[below]{$v$} -- (0,0)node[below left]{0} -- (1,0) node[above] {$1$}-- (1.1,0)node[right]{$u$} ;
%\draw (-.05,0)--(0,0)--(0,.02);
%\draw[domain=0:0.25,smooth] plot (\x, {g(\x)}) ;
%\draw[thick, domain=0.25:1] (0.25,0) -- plot (\x, {g(\x)}) -- cycle ;
%\draw(0.25,0)node[above]{$U(x)$}-- (0.25,{g(0.25)});
%\draw[dashed](0.5,0)node[above]{$U(x_0)$} -- (0.5,{g(0.5)});
%\pgfmathsetmacro{\cc}{0.02}
%\pgfmathsetmacro{\ccc}{0.05}
%\foreach \x in {2/8,3/8,4/8,5/8,6/8,7/8,5/16,7/16,9/16,11/16,13/16,15/16} 
%{
%        \draw[->,thick,dashed]  (\x, {g(\x)}) -- ++({8*-.5*\cc*dG(\x)},-1*\cc);
%        \draw[->,thick]  (\x, {g(\x)}) -- ++({(9*\ccc*-.5)*G(\x)},{1.2*\ccc*G(\x)});
%}
%\draw[->,thick,dashed]  (1, {g(1)}) -- ++({9*-.5*\cc*dG(1)},-1.2*\cc);
%\node[below] at (0.57,-0.08) {$\hat \xi$};
%\node[below] at (0.665,-0.13) {$\hat n$};
%\node[rotate=-50,font=\scriptsize] at (0.1,-0.036) {$v=-sg(u)$};
%\end{tikzpicture}
\caption{The region $D_x$ in the case $x\ge x_0$ (so that $U(x)\le U(x_0)$).}
\end{figure}
%$$$$$$$$$$$$$$$$$$$$$$$$$$$$$$
Consider the curve $\{(U(y),V(y))\}$, with $V:=U'$. By \eqref{3.2}, $(U',V')=(V,-c_\al V-g(U))$. Notice that the vector $\hat \xi:=(v,-c_\al v-g(u))$ is pointing inside  $D_x$ when $u\in(U(x),1)$ and $v=-sg(u)$. Indeed, the vector 
\[ \hat n:=(-sg'(u),-1) \]
is an outer normal to $D_x$, and $v=-sg(u)$ gives
\[ \hat \xi =g(u)(-s, c_\al s-1). \]
Since $g>0$ and $g'$ is decreasing on $(0,1)$ , $u\in( U(x),1)$ and  \eqref{5.8} now yield
\[
\hat n\cdot\hat \xi=g(u)[g'(u)s^2-c_\al s+1]
< g(u)[g'(U(x))s^2-c_\al s+1]
=0.
\]

As a consequence, if $(U(y_0),V(y_0))\in D_x$ for some $y_0< x$, then $(U(y),V(y))\in D_x$ for all $y\in [y_0,x)$. Or equivalently,  if $(U(y_0), V(y_0))\notin  D_x$ for some $y_0<  x$, then $(U(y),V(y))\notin D_x$ for all $y\le y_0$. In this latter case we have 
\begin{equation}\lb{5.9}
V(y)<-s g(U(y))
\end{equation}
for all $y\le y_0$. From \eqref{3.2}, \eqref{5.9}, \eqref{5.8}, and $g'(U(x))>0$ it follows that 
\[
V'(y)=-c_\al V(y)-g(U(y))
> (c_\al s-1) g(U(y))
= g'(U(x))s^2 g(U(y))
>0
\]
for all $y\le y_0$.  But then $U'(y_0)=\int_{-\infty}^{y_0} V'(y)dy> 0$, a contradiction. 

Thus we must have $(U(y_0),V(y_0))\in D_x$ for all $y_0< x$, which yields $V(x)\ge -sg(U(x))$ by continuity. This is precisely \eqref{5.1z}, proving the lemma for $x>x_0$.

We actually proved $-U'(y_0)\le q(g'(U(x)))g(U(y_0))$ whenever $y_0\le x$ and $x\ge x_0$.  Taking $x:=x_0$ and renaming $y_0$ to $x\,(\le x_0)$, this becomes $-U'(x)\le q(0)g(U(x))$ for $x\le x_0$.  But this is again \eqref{5.1z} because for $x\le x_0$ we have $g'(U(x))\le 0$, so $q(g'(U(x)))=q(0)$.
%Now fix any $x\le x_0$. Then $g'(U(x))\le 0$. Choose $y>x_0\ge x$, and so $g'(U(y))>0$ by increasing of $g'(U)$. By previous result, $(U(x),V(x))\in D_y$. So $-U'(x)<q(g'(U_y)))g(U(x)))$. This is true for all $y>x_0$. Let $y\searrow x_0$. By continuity of $q$, $q(g'(U(y)))\to q(0)=1/c_\al$. So, 
%\[ -U(x)\le \frac 1{c_\al} g(U(x))=q(g'(U(x)))g(U(x)). \]
%This proves the latter case. 
\end{proof}

%$$$$$$$$$$$$$$$$$$$$$$$$$$$$$$$$$$$$$$$
\section{Proof of Lemma 3.2 (Improved Harnack-Type Estimate for $\phi_\la$)} \lb{S6}
%$$$$$$$$$$$$$$$$$$$$$$$$$$$$$$$$$$$$$$$

Let us drop the subscript $\la$ in $\phi_\la$.  Define
%The technique we deploy to achieve a better Harnack estimate for $\phi_\la$, the eigenfunction of $H+a(x)$ (as in \eqref{2.2}, \eqref{1.6}), is comparison with another new function $\psi=\psi_\la$, where 
\begin{equation}
\ka(x):=H\left[\frac {|x|}2\right]=\frac 12 \int _{-\de}^\de J(y)(|x-y|-|x|)dy,
\end{equation}
which is continuous, even (because $J$ is), and supported in $[-\delta,\delta]$.  We also have 
\begin{equation}\lb{6.3}
0\le \ka\le \frac{\de^2}{2}J.
\end{equation}
To show this, observe that $\kappa=H[x_+]$, where $x_+:= \max\{x,0\}$. So for $x\in[-\de,0]$, 
\[
\ka(x)=\int_{-\de}^{x} J(y)(x-y)dy \in \left[ 0, \int_{-\de}^{x} J(x)(x-y)dy \right] \subseteq \left[ 0, \frac{\de^2}{2}J(x) \right]
\]
because $J$ is even and non-decreasing on $\mbr ^-$.  Since $\ka$ is also even and vanishes outside $[-\del,\del]$,  \eqref{6.3} follows.

We will first prove an estimate as in the lemma for the function
\begin{equation}\lb{6.1}
\psi:= ||\ka||_{L^1}^{-1} (\ka*\phi),
\end{equation}
and then show that $\phi\psi^{-1}$ is close to 1 when $\la-a_->0$ is small. 
The motivation for introducing the function $\psi$ is the fact that 
\beq\lb{6.1z}
(\ka*\varphi)''=H\varphi
\eeq
 for any continuous function $\varphi$, showing that
\begin{equation}\lb{6.4}
\psi''=||\ka||_{L^1}^{-1} H\phi=||\ka||_{L^1}^{-1}(\la-a(x))\phi
\end{equation}
(which is small when $\la-a_-$ is small).  

Identity \eqref{6.1z} should hold because for $m:= \frac 12 |x|$ we have $m''=\de_0$ (the delta function at 0) in the sense of distributions, so formally $\ka ''=H[m'']=H\de_0=J-\de_0$.  To prove \eqref{6.1z}, let $0\le \eta\le 1$ be a smooth bump function around $x$ with $\eta=1$ on $[x-2\de,x+2\de]$, and $\eta=0$ outside $[x-4\de,x+4\de]$. If $\tilde \varphi:=\varphi\eta$, then $\ka*\varphi=\ka*\tilde{\varphi}$ and $H\varphi=H\til\varphi$ on $[x-\de,x+\de]$.  We have 
\[
\ka*\tilde \varphi=(J*m-m)*\tilde{\varphi}=J*m*\tilde \varphi-m*\tilde\varphi
\]
 because $\tilde\varphi$ and $J$ are compactly supported.   Since
\[
 \int_\bbR\int_\bbR m(x-y)\til \varphi(y)\tht''(x) dydx = -\int_\bbR \til \varphi(y) \int_\bbR m'(x-y)\tht'(x) dxdy = \int_\bbR \til \varphi(y) \tht(y)dy
 \]
for any $\tht\in C_0^\infty(\bbR)$, we see that $(m*\til\varphi)'' = \til \varphi$ in the distributional sense.  Similarly, we have $(J*m*\til\varphi)'' = J*\til \varphi$, and both equalities hold pointwise because the right-hand sides are continuous functions.
Thus $(\ka*\tilde{\varphi})''=H\tilde \varphi$, so $(\ka*{\varphi})''(x)=(H\varphi)(x)$. This holds for any $x\in\bbR$, yielding \eqref{6.1z}.

The properties of $\phi$ and \eqref{6.3} show $\psi>0$ and $\lim_{x\to\infty} \psi(x)=0$.  Then \eqref{6.4} and $\la>a_+$ show  $\psi'<0$.  We also claim the following.

\begin{lemma}\lb{L.6.1}
There is $m_{s}=m_{s}(J)$ such that $\lim_{s \searrow 0}m_{s}=0$ and $|\psi '(x)|\le m_{\la-a_-}\psi(x)$.  In particular, $e^{-m_{\la-a_-}\de}\psi(x)\le \psi(x-y)\le e^{m_{\la-a_-}\de}\psi(x)$ whenever  $|y|\le\de$.
\end{lemma}

\begin{proof}
With $C=C_J(1+\la-a_-)$ from Lemma \ref{L.2.2}, and from the remark following it,we have 
%$\psi(x)\ge e^{-C\de}\phi(x)$ because 
\beq\lb{6.5}
\frac{e^{-C\de}}{1+\la -a_-}\phi(x) \le \psi(x)
%=\frac{1}{||\ka||_{L^1}}\int_{-\de}^{\de}\ka(y)\phi(x-y)dy 
\le (1+\la -a_-)e^{C\de}\phi(x).
\eeq
%For convenience, we write $C_{\la-a_-}\equiv e^{C\de}$. Observe that from \eqref{2.8}, $C_{\la-a_-}\to e^{C_J+a'_+}$ as $\la-a_-\to 0$. 
%Taking second derivative of $\psi$, and using 
Then \eqref{6.4} and \eqref{6.5} give 
\[ \psi ''(x)\le \frac{e^{C\de}(\la-a_-)(1+\la -a_-)}{||\ka||_{L^1}} \psi(x),\]
which then implies 
\[ -\psi '(x)\le \sqrt{\frac{e^{C\de}(\la-a_-)(1+\la -a_-)}{||\ka||_{L^1}}}\psi(x). \]
To see the latter, let $\mu$ be the constant on the right-hand side of the above inequality. Recall that $\psi ''\le \mu^2 \psi $ and  $\psi,\psi'' >0>\psi' $. Thus $Q:=-\psi '/\psi>0 $ satisfies $Q'\ge Q^2 -\mu^2$.  So if  $Q(x_0)> \mu$ for some $x_0\in \mbr $, then $Q'>0$ on $(x_0,\infty)$. Together with $Q'\ge Q^2-\mu^2$ this shows that $Q$ must blow up at some $x_1\in (x_0,\infty)$,  a contradiction. Thus $Q\in (0, \mu]$, as claimed. 

So we can let $m_{\la-a_-}$ be this $\mu$, and $\lim_{s\searrow 0}m_s =0$ is obvious.
\end{proof}

% Clearly, $m_{\la-a_-}\to 0$ as $C_{\la-a_-}$ does not blow up at $\la-a_-\to 0$. This concludes (2) of the lemma. 

\begin{lemma}\lb{L.6.2}
There are $l_s =l_s (J) < L_s =L_s (J)$ such that  $\lim_{s\searrow 0} l_s =\lim_{s\searrow 0}L_s=1$ and  $l_{\la -a_-} \phi(x)\le \psi(x)\le L_{\la-a_-} \phi(x)$. 
\end{lemma}

%% % % % figure % % %
%\begin{figure}[h]
%\centering
%\begin{tikzpicture}[yscale=1.5,domain=4:-4]
%\draw[->] (-4.2,0) -- (4.2,0)node[right]{$x$};
%\draw (0,-0.2) -- (0,3);
%\draw[smooth,thick] plot (\x,{exp(-.2*\x)}) node[left] {$\psi(x)$};
%\draw[smooth, dashed, thick, domain=3.5:-3.6] plot (\x,{(1+0.2*cos(250*\x))*exp(-.2*\x)})node[left]{$\phi(x)$};     
%\end{tikzpicture}
%\caption{The graph of $\phi$ and $\psi$. }
%\end{figure}
%% % % % % % %

\begin{proof}
Let $M_{s}:=e^{-m_{s}\de}$, with $m_{s} $ from Lemma \ref{L.6.1}, and define
\begin{equation}
\mu := \inf_{x\in\mbr}\frac{\psi(x)}{\phi(x)},\qquad 
\nu := \sup_{x\in\mbr} \frac{\psi(x)}{\phi(x)}.
\end{equation}
We have  $0< \mu\le \nu <\infty$ by \eqref{6.5}.
% Lemma follows if there is $l_{\la -a_-}\le \mu \le \nu \le L_{\la -a_-}$ with $\lim_{s\searrow 0} L_s=\lim_{s\searrow 0}l_s=1$. By definition of $\nu $, given $\ep>0$, choose 
Given any $\eps>0$, let $x_0$ be such that 
\begin{equation} \lb{6.8z}
(1-\eps)\nu \le \frac{\psi(x_0)}{\phi(x_0)}\le \nu .
\end{equation} 

If $|y|\le \de $, then by Lemma \ref{L.6.1},
\begin{equation}\label{6.8}
\frac{\phi(x_0-y)}{\phi(x_0)}=\frac{\phi(x_0-y)\psi(x_0-y)\psi(x_0)}{\psi(x_0-y)\psi(x_0)\phi(x_0)}\ge (1-\eps ) M_{\la-a_-}  .
\end{equation}
Thus we find that 
\begin{eqnarray*}
\int_{-\de}^{\de}J(y)[\phi(x_0-y)-\phi(x_0)]_+dy&=&H\phi(x_0)+\int_{-\de}^{\de } J(y)[\phi(x_0-y)-\phi(x_0)]_-dy\\
&\le&(\la-a_-)\phi(x_0)+[1- (1-\eps) M_{\la-a_-} ]\phi(x_0).
\end{eqnarray*}
So by the definition of $\psi$ and \eqref{6.3},
\begin{eqnarray*}
\psi(x_0)&=&\phi(x_0)+\frac{1}{||\ka||_{L^1}} \int_{-\de}^{\de }\ka(y)[\phi(x_0-y)-\phi(x_0)]dy\\
&\le&\phi(x_0)+ \frac{\de^2}{2||\ka||_{L^1}} \int_{-\de}^{\de }J(y)[\phi(x_0-y)-\phi(x_0)]_+dy\\
&\le& \phi(x_0)+\frac{\de^2}{2||\ka||_{L^1}} [\la-a_-+1- (1-\eps) M_{\la-a_-} ] \phi(x_0).
\end{eqnarray*}
Hence \eqref{6.8z} shows 
\[ (1-\eps)\nu  \le 1+ \frac{\de^2}{2||\ka||_{L^1}} [\la-a_-+1- (1-\eps)M_{\la-a_-} ].\]
Taking $\eps\to 0$ yields 
\[ \nu \le 1+ \frac{\de^2}{2||\ka||_{L^1}} [\la-a_-+1-M_{\la-a_-}]=:L_{\la -a_-} = L_{\la -a_-}(J),\]
and $\lim_{s\searrow 0}L_s=1$ follows from the same for $M_s$, which is due to Lemma \ref{L.6.1}.

A similar argument, using $\psi(x_0)\phi(x_0)^{-1}\le (1+\eps)\mu$ to show
$\phi(x_0-y)\phi(x_0)^{-1} \le (1+\eps ) M_{\la-a_-}^{-1}$ 
for  $|y|\le \de $ 
and then
\begin{eqnarray*}
-\int_{-\de}^{\de}J(y)[\phi(x_0-y)-\phi(x_0)]_-dy&=&H\phi(x_0)-\int_{-\de}^{\de } J(y)[\phi(x_0-y)-\phi(x_0)]_+dy\\
&\ge&(\la-a_+)\phi(x_0)-[(1+\eps) M_{\la-a_-}^{-1} -1 ]\phi(x_0),
\end{eqnarray*}
shows (recall that $\la>a_+$)
\[ \mu\ge 1-\frac{\de^2}{2||\ka||_{L^1}} [M_{\la-a_-}^{-1}-1] =:l_{\la -a_-} = l_{\la -a_-}(J). \]
Again, $\lim_{s\searrow 0}l_s=1$ is immediate. 
\end{proof}

%\begin{proof}[Proof of Lemma \ref{L.3.2}]
To prove  Lemma \ref{L.3.2}, it suffices to show $\phi (x-y)\le C_{\la-a_-} \phi(x)$ whenever $|y|\le \de $, where $C_s=C_s(J)$ and $\lim _{s \searrow 0}C_s =1$ (then $\gamma_s:=C_s-1$). By Lemmas \ref{L.6.2} and \ref{L.6.1}, 
\[ \phi(x-y)\le l^{-1}_{\la-a_-} \psi (x-y)\le l^{-1}_{\la-a_- }e^{m_{\la-a_-}\de} \psi(x)\le l^{-1}_{\la -a_-}e^{m_{\la-a_-}\de}L_{\la -a_-}\phi(x). \]
Hence we set $C _s := l_{s} ^{-1}L_{s} e^{m_{s}\de}$, and the proof is finished.
% As the conclusion of lemma \ref{L.6.1} and \ref{L.6.2}, $C _s \to 1$ as $s\searrow 0$. This proves the lemma.
%\end{proof}

%%%%%%%%%%%%%%%%%%%%%%%%%%%%%%%%

\end{document}